\documentclass[12pt,oneside]{amsart}
\usepackage{amsmath}
\usepackage{amsthm}
\usepackage{amsfonts}
\usepackage{amssymb}
\newtheorem{theorem}{Theorem}
\newtheorem{lemma}[theorem]{Lemma}
\newtheorem{prop}[theorem]{Proposition}
\newtheorem{corollary}[theorem]{Corollary}

\theoremstyle{definition}

\theoremstyle{remark}
\newtheorem{remark}[theorem]{Remark}
\newcommand{\BB}{{\mathbb B}}
\newcommand{\DD}{{\mathbb D}}
\newcommand{\OO}{{\mathcal O}}

\newcommand{\EE}{{\mathbb E}}

\newcommand{\CC}{{\mathbb C}}

\newcommand{\TT}{{\mathbb T}}

\DeclareMathOperator{\Aut}{Aut} \DeclareMathOperator{\id}{id}

\renewcommand{\phi}{\varphi}


\begin{document}

\title[Extremal holomorphic maps]{Extremal holomorphic maps in special classes of domains}

\address{Institute of Mathematics, Faculty of Mathematics and Computer Science, Jagiellonian
University,  \L ojasiewicza 6, 30-348 Krak\'ow, Poland\\
\indent D\'epartement de math\'ematiques et de statistique, Pavillon Alexandre�Vachon, 1045 av. de la M\'edecine, Universit\'e Laval, Qu\'ebec (Qu\'ebec), Canada, G1V 0A6.}
\author{\L ukasz Kosi\'nski}\email{lukasz.kosinski@im.uj.edu.pl}

\address{Institute of Mathematics, Faculty of Mathematics and Computer Science, Jagiellonian
University,  \L ojasiewicza 6, 30-348 Krak\'ow, Poland}
\author{W\l odzimierz Zwonek}\email{wlodzimierz.zwonek@im.uj.edu.pl}
\thanks{The first author is partially supported by the Polish Ministry of Science and Higher Education grant Iuventus Plus IP2012 032372.\\
The second author is partially supported by the grant of the Polish National Science Centre no. UMO--2011/03/B/ST1/04758.}

\maketitle

\begin{abstract}
In the paper we discuss three different notions of extremal holomorphic mappings: weak $m$-extremals, $m$-extremals and $m$-complex geodesics.
We discuss relations between them in general case and in special cases of the unit ball, classical Cartan domains, symmetrised bidisc and tetrablock.
In particular, we show that weak $3$-extremal maps in the symmetrised bidisc are rational thus giving the (partial) answer to a problem posed in a recent paper by J. Agler, Z. Lykova and N. J. Young  (\cite{Agl-Lyk-You 2013}).
\end{abstract}

\section{Introduction}

Throughout the paper $\mathbb D$ will always denote the unit disc in the complex plane.

Let $D$ be a domain in $\mathbb C^n$. Let $m\geq 2$ and let $\lambda_1,\ldots,\lambda_m\in \mathbb D$ be distinct (\textit{distinct } means in the paper
pairwise distinct) and $z_1,\ldots,z_m\in D$. Following \cite{Agl-Lyk-You 2013}
we say that the interpolation data
\begin{equation}
 \lambda_j\mapsto z_j,\ \mathbb D\to D,\quad j=1,\ldots,m,
\end{equation}
are {\it extremally solvable} if there is a map $h\in\OO(\mathbb D,D)$ such that $h(\lambda_j)=z_j$, $j=1,\ldots,m$,
and there is no $f\in\OO(\overline{\mathbb D},D)$ (i.e. $f$ is holomorphic on some neighborhood of $\overline{\mathbb D}$ and its image lies in $D$)
such that $f(\lambda_j)=z_j$, $j=1,\ldots,m$.

We say that $h\in\OO(\mathbb D,D)$ is {\it $m$-extremal} if for all choices of $m$ distinct points $\lambda_1,\ldots,\lambda_m\in \mathbb D$
the interpolation data
\begin{equation}
 \lambda_j\mapsto h(\lambda_j),\ \mathbb D\to D,\quad  j=1,\ldots,m,
\end{equation}
are extremally solvable. Note that if $h$ is $m$-extremal then it is $(m+1)$-extremal.

Generally, the fact that for a fixed $m$ the interpolation data are extremally solvable for some $\lambda_1,\ldots,\lambda_m$
does not imply that the interpolation data are extremally
solvable for all other $m$ points $\mu_1,\ldots,\mu_m$. This is already not the case generally for $m=2$. In other words extremals with respect to the Lempert function for some pair of points need not be extremal
for the Lempert function for any pair of points (see Section~\ref{Ident} for a definition of the Lempert function). In particular, there are domains, e.g. the annulus in the complex plane, possessing no $2$-extremals. For basic properties of the Lempert function (and other holomorphically invariant functions) that we shall use we refer the Reader to \cite{Jar-Pfl 1993}.

Therefore, it is natural to introduce a weaker notion of $m$-extremal map which is equivalent with the notion of extremal in the sense of Lempert when $m=2$. This may be done as follows: if an analytic disc $f:\mathbb D\to D$ and fixed points $\lambda_1,\ldots,\lambda_m\in \mathbb D$ are such that the problem $\lambda_j\mapsto f(\lambda_j)$ is extremally solvable, then we shall say that $f$ is a {\it weak $m$-extremal with respect to $\lambda_1,\ldots,\lambda_m$}. Naturally, $f$ will be said to be a {\it weak $m$-extremal} if it is a weak extremal with respect to some $m$ distinct points in the unit disc.

Of course $m$-extremals are weak $m$-extremals for any system of $m$ distinct points $\lambda_1,\ldots,\lambda_m\in\mathbb D$.
In general, the class of weak $m$-extremals is strictly bigger than the class of $m$-extremals (as already mentioned even if $m=2$ with $D$ being for instance the annulus). Similar problems concerning some kind of $m$ extremality in several variable context were considered for instance in \cite{Ama-Tho1994}, \cite{Edi1995}.

Certainly these two notions coincide in the case of domains for which the assertion of the Lempert theorem holds. Recall this theorem in the form it would be convenient for us (see \cite{Lem 1981} and \cite{Lem 1982}).

\begin{theorem}\label{thm:lempert} Let $D$ be a bounded convex or smooth strongly linearly convex domain in $\mathbb C^n$. Then for any $w,z\in D$, $w\neq z$ there are a holomorphic mapping $f:\mathbb D\to D$ such that
 $w$ and $z$ lie in the image of $f$ and a holomorphic function $F:D\to\mathbb D$ such that $F\circ f$ is the identity $\operatorname{id}_{\mathbb D}$.
\end{theorem}

In fact the function $f$ in the above result is extremal for the Lempert function of $w$ and $z$ whereas $F$ is extremal for the Carath\'eodory distance of $w$ and $z$. Let us call the function $F$ \textit{a left inverse of $f$}.
Recall also that the Lempert theorem holds for the symmetrised bidisc and the tetrablock (see \cite{Agl-You 2004}, \cite{Cos 2004b} and \cite{Edi-Kos-Zwo}).

Note also that (weak) $m$-extremals (for $\lambda_1,\ldots,\lambda_m$) are (weak) $(m+1)$-extremals (for $\lambda_1,\ldots, \lambda_{m+1}$ with arbitrarily chosen $\lambda_{m+1}\in\mathbb D$ distinct from $\lambda_1,\ldots,\lambda_m$).

The existence of left inverses in the Lempert Theorem suggests another notion of extremal mappings.
Namely, we generalize the notion of a complex geodesic (see e.g. \cite[Section~8.2]{Jar-Pfl 1993}) as follows.

Let $f:\DD\to D$ be a holomorphic mapping, $m\geq 2$. We say that $f$ is an {\it $m$-complex geodesic} if there is a function $F\in\OO(D,\DD)$
such that $F\circ f$ is a non-constant Blaschke product of degree at most $m-1$.
Note that $2$-complex geodesics are simply complex geodesics and any $m$-complex geodesic is an $(m+1)$-complex geodesic.

The aim of the paper is to try to understand the relations between the notions of $m$-extremals, weak $m$-extremals and $m$-complex geodesics in special classes of domains
(convex ones, classical Cartan domains, the unit ball, symmetrised bidisc and tetrablock).

We also see that in some class of domains (containing for example classical Cartan domains) the notions of weak $m$-extremals and $m$-extremals are equivalent (Proposition~\ref{prop}). Clearly, in the polydisc all $m$-extremals are $m$-complex geodesics. This is not the case for the the Euclidean unit ball. We show that there are $4$-extremals in the unit ball which are not $4$-complex geodesics (Proposition~\ref{prop:four-extremals-are-not-geodesics}).

Finally we present a new method for describing (weak) $m$-extremals in the symmetrised bidisc. In our approach the crucial role is played by the geometry of the tetrablock - the domain that, similarly to the symmetrised bidisc, arises naturally in the control-engineering problems. Then some arguments allow us to reduce the problem to already investigated classical domains. The results giving the description of weak $m$-extremals are given in Theorem~\ref{lem1} (arbitrary $m$) and in Theorems~\ref{thla} and \ref{thlb} ($m=3$).

In particular, we show that all weak 3-extremals in the symmetrised bidisc are rational and map $\mathbb T$ into the Shilov boundary . As a corollary we get that all $3$-extremals in the symmetrised bidisc
are rational of degree at most $4$ (Theorem~\ref{thm:three-extremals-are-rational}) which gives answer to a problem posed in \cite{Agl-Lyk-You 2013} in case $m=3$ (this case was also studied in \cite{Agl-Lyk-You 2013}).

Finally, Proposition~\ref{prop:identity-is-extremal} shows that the identity is $m$-extremal (in a more general sense -- see Section~\ref{Ident}) 
which answers a problem also posed in \cite{Agl-Lyk-You 2013}.

Here is some notation: for $\alpha\in\mathbb D$ let $m_\alpha(\lambda)= \frac{\alpha - \lambda}{1-\bar \alpha \lambda}$, $\lambda\in\mathbb D$, be a Blaschke factor. Moreover, $\mathbb T$ is the unit circle in the complex plane and $\mathbb C^{k\times l}$ stands for the space of $k\times l$ complex matrices. We shall denote by $\Aut(D)$ the group of holomorphic automorphisms of a domain $D$ of $\mathbb C^n$. Moreover, $\partial_s D$ denotes the Shilov boundary with respect to the algebra $\mathcal O(D)\cap \mathcal C(\overline D)$ of a bounded domain $D$ of $\mathbb C^n$.

The authors would like to thank the anonymuous referee for many remarks which improved the presentation of the paper.

\section{Results on (weak) $m$-extremals and $m$-complex geodesics. General case and classical Cartan domains}

We start with some basic properties and relations between different notions of extremal mappings.

\begin{prop}
Let $f:\DD\to D$ be a holomorphic mapping. Assume that $F:D\to\DD$ is such that $F\circ f$ is a Blaschke product of degree $m$. Then
$f$ is $(m+1)$-extremal.
\end{prop}
\begin{proof}[Proof] Recall that it is well-known that the Blaschke product $F\circ f:\mathbb D\to\mathbb D$ is $(m+1)$-extremal (see e.g. \cite{Pick}). Suppose that $f$ is not an $(m+1)$-extremal.
Then there is a holomorphic  mapping $g:\DD\to D$ with $g(\DD)\subset\subset D$ such that for some $m+1$ distinct points
$\lambda_1,\ldots,\lambda_{m+1}$ we have $g(\lambda_j)=f(\lambda_j)$, $j=1,\ldots,m+1$. Then $(F\circ g)(\lambda_j)=
(F\circ f)(\lambda_j)$, $j=1,\ldots,m+1$ and $(F\circ g)(\DD)\subset\subset \DD$ which contradicts the $(m+1)$-extremality of $F\circ f$.
\end{proof}

\begin{corollary} Let $f\in\OO(\DD,D)$, $F\in\OO(D,\DD)$. Assume that the function $B:=F\circ f\in\OO(\DD,\DD)$ is a Blaschke product of degree $m$
and $B_1$ is a Blaschke product of degree $k$. Then the function $\DD\owns\lambda\mapsto f(B_1(\lambda))\in D$ is an $(mk+1)$-extremal.
\end{corollary}

The three notions introduced in the preliminary section have clear relations: an $m$-complex geodesic is an $m$-extremal and an $m$-extremal is a weak $m$-extremal (for any system of $m$ pairwise points). Recall that a weak $m$-extremal need not be an $m$-extremal. Already in the case $m=2$ the example of the holomorphic covering of the annulus is a weak $2$-extremal for (some) pairs of points and yet it is not $2$-extremal for all pairs of points - it follows from the fact that the Lempert function of two different points from the annulus is equal to the Poincar\'e distance of some two points from the unit disc which belong to the preimages of the given points of a holomorphic covering of the annulus - consult \cite[Chapter V]{Jar-Pfl 1993} for details. Recall also that it follows from the celebrated Lempert theorem that all weak $2$-extremals are $2$-complex geodesics. In our paper we get some results on the lacking implications in special classes of domains.

We start with the analysis of properties of extremals in classical domains. We shall focus on classical Cartan domains of the first, second and third type denoted by $\mathcal R^{n,m}_I$ and $\mathcal R^n_{II}$ and $\mathcal R^n_{III}$, respectively.
In particular, $\mathbb B_n=\mathcal R^{1,n}_I$ is the unit Euclidean ball. Note that results obtained here work for other classes of domains, not necessarily symmetric, like the Lie ball. In this section we shall use the letter $\mathcal R$ to denote any of these domains or their Cartesian products. We just demand $\mathcal R$ to be a bounded, convex and balanced domain in $\mathbb C^n$ whose group of holomorphic automorphisms acts transitively.

Since classical domains of $\mathbb C^{2\times 2}$ play a crucial role in the paper,
we put $\mathcal R_I:= \mathcal R^{2,2}_I$ and $\mathcal R_{II}:=\mathcal R^2_{II}$. For definitions and basic properties of the Cartan domains that we shall use in the paper we refer the Reader to \cite{Har} and \cite{Hua}.

\begin{remark}\label{remark:balanced-extremals} Assume that $D$ is a bounded balanced pseudoconvex domain and assume that $f:\mathbb D\to D$ is a weak $m$-extremal for $\lambda_1,\ldots,\lambda_m$ such that $f(0)=0$.
We may write $f(\lambda)=\lambda \psi(\lambda)$, $\lambda\in \mathbb D$ for some analytic disc $\psi$. Let $\mu_D$ denote the Minkowski functional of $D$. Pseudoconvexity of the domain $D$ guarantees that $\log \mu_D$ is plurisubharmonic (see e.g. \cite[Proposition 2.2.22]{Jar-Pfl 2000}). Note that $\log \mu_D\circ f<0$ on $\mathbb D$, whence $\limsup_{|\lambda|\to 1}\log \mu_D (\psi(\lambda))\leq 0$. As a consequence of the maximum principle for subharmonic functions we get that $\log \mu_D\circ \psi\leq 0$ on $\mathbb D$. In particular, the image of $\psi$ lies in $\bar D$. The maximum principle for subharmonic functions implies that if $\log \mu_D(\psi(0))=1,$ then $\log \mu_D\circ \psi\equiv 1$ and consequently the image of $\psi$ lies entirely in $\partial D$. Otherwise, $\mu_D\circ \psi<1$, whence $\psi$ is an analytic disc in $D$.

If the first possibility holds, i.e. $\mu_D\circ \psi\equiv 1$
then $\mu_D(f(\lambda))=|\lambda|$, which easily implies that $f$ is a weak $2$-extremal for $0$ and arbitrary $\lambda\in\mathbb D\setminus\{0\}$.

If the second possibility holds, i. e. $\psi:\DD\to D$, then in the case $\lambda_j\neq 0$ for any $j$ the mapping $\psi$ is a weak $m$-extremal for $\lambda_1,\ldots,\lambda_m$.
If, on the other hand,  $\lambda_m=0$ and $m\geq 3$, then $\psi$ is a weak $(m-1)$-extremal for $\lambda_1,\ldots,\lambda_{m-1}$.
\end{remark}

As a consequence of the above remark we get the following procedure of producing 'new' extremals from the existing ones.

\begin{remark}\label{remark:balanced-extremals2}

Let $f:\mathbb D\to D$, where $D$ is a balanced pseudoconvex domain, be a weak $m$-extremal for distinct points $\lambda_1,\ldots,\lambda_m\in\DD$.
Then the function $g$ given by the formula
$g(\lambda):=m_{\lambda_{m+1}}(\lambda)f(\lambda)$, $\lambda\in\DD$, where $\lambda_{m+1}\neq\lambda_j$, $j=1,\ldots,m$, is a weak $(m+1)$-extremal for $\lambda_1,\ldots,\lambda_{m+1}$.
\end{remark}

Note that in the above remarks we made extensive use of the pseudoconvexity of the balanced domain $D$ -- this is equivalent to the fact that the Minkowski functional $\mu_D$ is logarithmically plurisubharmonic (and certainly homogeneuous). The class of such domains obviously contains the balanced convex domains. Actually, we make use of the above remarks only in that type of domains; however, the general case is interesting, too; so we decided to leave the proof in the more general setting as it is essentially the same in both cases. 

\begin{prop}\label{prop:proper-extremals} Let $f:\DD\to\mathcal R$ be a weak $m$-extremal for some $m$ distinct points in a classical Cartan domain $\mathcal R$, $m\geq 2$. Then $f$ is proper.
\end{prop}
\begin{proof} We proceed inductively with respect to $m$. In case $m=2$ it follows from the transitivity of the group of automorphisms and the fact that any weak $2$-extremal $f:\DD\to\mathcal R$ with $f(0)=0$ is such that
$\mu_{\mathcal R}(f(\lambda))=|\lambda|$, $\lambda\in\DD,$ where $\mu_{\mathcal R}$ is the Minkowski functional of $\mathcal R$.

Let now $f$ be a weak $m$-extremal for $\lambda_1,\ldots,\lambda_m$ in $\mathcal R$, $m\geq 3$. Due to transitivity of $\mathcal R$ we may assume that $\lambda_m=0$ and $f(0)=0$.
To finish the proof it is sufficient to make use of Remark~\ref{remark:balanced-extremals} and inductive assumption for $\psi$ as defined in Remark~\ref{remark:balanced-extremals} (in the case $\psi$ does not lie in the boundary of $\mathcal R$).
\end{proof}

\begin{remark} Recall (see e.g. \cite{Har, Hua}) that the classical domains are homogeneous (i.e. their group of holomorphic automorphisms act transitively), balanced and convex.
\end{remark}

\begin{prop}\label{prop} Any weak $m$-extremal in any of the the classical Cartan domains $\mathcal R$ is an $m$-extremal.
\end{prop}

\begin{proof}
We will proceed inductively. For $m=2$ the assertion is a simple consequence of the Lempert Theorem.

So assume that $m\geq 3$, any $(m-1)$-weak extremal is an $(m-1)$-extremal and let $f$ be a weak $m$-extremal with respect to $\lambda_1,\ldots,\lambda_m$. Composing $f$ with a M\"obius map we may assume that $\lambda_m=0$.
Thanks to transitivity of $\Aut(\mathcal R)$ one may moreover assume that $f(0)=0.$ Then $f(\lambda)=\lambda \varphi(\lambda)$, where $\varphi$ is an analytic disc in $\overline{\mathcal R}$. 

If $\varphi(0)\in \partial \mathcal R$, then $f$ is a weak $2$-extremal for $0$, $\lambda$ for any $\lambda\in\mathbb D\setminus\{0\}$ (see Remark~\ref{remark:balanced-extremals}) and thus, in view of (for instance) the Lempert theorem $2$-extremal.

In the other case, $\varphi$ is an analytic disc in $\mathcal R$ and it is a weak $(m-1)$-extremal for some system of points and thus, due to the inductive assumption $(m-1)$-extremal in $\mathcal R$. Take any points $\sigma_1,\ldots ,\sigma_m\in \mathbb D$. We claim that $f$ is a weak extremal with respect to them. If $\sigma_j=0$ for some $j$ (without loss of generality assume that $j=m$)
then the fact that $f$ would not be a weak $m$-extremal for $\sigma_1,\ldots,\sigma_m$ would deliver an analytic disc lying relatively compactly in $\mathcal R$ coinciding with $\varphi$ at points $\sigma_1,\ldots,\sigma_{m-1}$ contradicting the $(m-1)$-extremality of $\varphi$. So suppose that $\sigma_j$ does not vanish. Seeking a contradiction assume that one may find a holomorphic mapping $g:\overline{\mathbb D} \to \mathcal R$
such that $g(\sigma_j)=f(\sigma_j)$, $j=1,\ldots,m$. Let $\Psi_a$ denote the automorphism of $\mathcal R$ such that $\Psi_a(a)=0,$ $\Psi_a(0)=a$.

Since $g(\sigma_1)= \sigma_1 \varphi(\sigma_1)$ we see that an analytic disc $\psi$ given by the formula
$\psi(\lambda) = \frac{1}{m_{\sigma_1}(\lambda)} \Psi_{\sigma_1 \varphi(\sigma_1)}( g(\lambda))$ maps
$\mathbb D$ into $\mathcal R$ (note that the image of $\psi$ cannot lie in $\partial \mathcal R$). We shall show that $\psi$ is a weak $(m-1)$-extremal in $\mathcal R$.
This would give a contradiction as $\psi(\bar{\mathbb D})\subset \mathcal R$. It suffices to show that $\chi:\lambda\mapsto \frac{1}{m_{\sigma_1}(\lambda)} \Psi_{\sigma_1 \varphi(\sigma_1)} (\lambda \varphi(\lambda))$ is a weak $(m-1)$-extremal, as it agrees with $\psi$ at points $\sigma_j$, $j=2,\ldots,m$. Note that $\chi(0)=\varphi(\sigma_1)$.

If $\chi$ were not an $(m-1)$-extremal then we would be able to find an analytic disc $\tilde \chi:\overline{\mathbb D} \to \mathcal R$ such that $\chi(0)=\tilde \chi(0)$ and $\chi(\sigma_j)=\tilde \chi (\sigma_j)$ for $j=3,\ldots,m$. Then $\tilde \varphi(\lambda):=\frac{1}{\lambda} \Psi_{\sigma_1 \varphi(\sigma_1)} (m_{\sigma_1}(\lambda) \tilde \chi(\lambda))$ would be well defined (we remove singularity at $0$) and would agree with $\varphi$ at $\sigma_3,\ldots, \sigma_m$. Moreover, it follows immediately from the definition that $\tilde \varphi(\sigma_1)=\varphi(\sigma_1).$ This gives a desired contradiction.

\end{proof}

Problem: Does a similar result hold for the symmetrised bidisc $\mathbb G_2$ or the tetrablock $\mathbb E$ if $m\geq 3$? Does a similar result hold for any convex domain?

\begin{remark}\label{remark:production-of-extremals-in-classical-domains} In the classical domains the fact that all weak extremals are extremals allows us to produce new extremals from the existing ones. For instance, let $f:\DD\to\mathcal R$ be
an $m$-extremal. Then, in view of Remark~\ref{remark:balanced-extremals2} the function $m_{\alpha}(\lambda)f(\lambda)$ is a weak $(m+1)$-extremal for $\lambda_1,\ldots,\lambda_m,\alpha$ for $m$ distinct points $\lambda_1,\ldots,\lambda_m$, $\alpha$ where $\alpha\neq\lambda_j$ and thus it is $(m+1)$-extremal.
Consequently, the function $B\cdot f$ where $B$ is a Blaschke product of degree $k$ is an $(m+k)$-extremal. In particular, the function $g$ given by the formula $g(\lambda)=\lambda^kf(\lambda)$, $\lambda\in\DD$, is an $(m+k)$-extremal. This observation will be used later.
\end{remark}

\section{$m$-complex geodesics and $m$-extremals in the unit Euclidean ball}
Let $f:\DD\to \BB_n$ be a $3$-extremal in the unit ball that is not a $2$-extremal, $n\geq 2$. Composing it with an automorphism of $\BB_n$ we may assume that $f(0)=0$. Let us write $f(\lambda) = \lambda g(\lambda),$ where $g$ is $2$-extremal in $\BB_n$. Composing $g$ with a unitary automorphism of $\BB_n$ we may additionally require that $g$ passes through points of the form $(a_1, b_1,0,\ldots, 0)$ and $(a_1, b_2,0,\ldots, 0)$, where $a_1\geq 0$. Since any $2$-extremal in $\BB_n$ is the image of $\lambda\mapsto (\lambda, 0,\ldots, 0)$ under an automorphism of $\BB_n$, making use of the description $\Aut(\BB_n)$ (see e.g. \cite[Appendix]{Jar-Pfl 1993}) we easily find that $g(\lambda) = (a_1, \sqrt{1-a_1^2} m(\lambda), 0,\ldots, 0)$, $\lambda\in \DD$, for some M\"obius map $m.$

Consequently, any $3$-extremal in the unit ball is, up to a composition with an automorphism of $\mathbb B_n$, of the form 
\begin{equation}
\lambda\mapsto (\lambda a_1,\sqrt{1-a_1^2}\lambda m(\lambda),0,\ldots,0),
\end{equation} 
where $m$ is a M\"obius map and $a_1\geq 0$. We do not know whether such 3-extremals are $3$-complex geodesics but we are able to show it at least in the case when $m$ is a rotation.

\begin{prop}\label{prop:three-extremals-are-geodesics} Let $f\in\OO(\DD,\BB_n)$, $n\geq 2$ be a $3$-extremal of the form $\lambda\mapsto (\lambda a_1,\sqrt{1-a_1^2}\lambda ^2,0,\ldots,0)$, where $a_1\geq 0$. Then $f$ is a $3$-complex geodesic.
\end{prop}
\begin{proof} It suffices to observe that the left inverse to the $3$-extremal $f$ (given by the formula $f(\lambda)=(a_1\lambda,\sqrt{1-a_1^2}\lambda^2,0,\ldots,0)$) may be chosen as follows:
\begin{equation}
F(z):=\frac{z_1^2}{2-a_1^2}+\frac{2\sqrt{1-a_1^2}}{2-a_1^2}z_2,\quad z\in\BB_n.
\end{equation}
It is simple to see that $F\in\OO(\BB_n,\DD)$ and $F(f(\lambda))=\lambda^2$, $\lambda\in\DD$.
\end{proof}

It is not clear from the first view why the left inverse in the above result is of the form as given above. In fact, the idea that resulted in that form will be more clear after the study
of the proof of the next result where we shall prove that there are $4$-extremals in the unit ball which are not $4$-complex geodesics. Note that making use of the procedures described in the preceing section one may relatively easily produce necessary form for $m$-extremals. Another way of finding the necessary form of
$m$-extremals (in more general domains called complex ellipsoids) was presented in \cite{Edi1995}.

\begin{prop}\label{prop:four-extremals-are-not-geodesics} Let $k\geq 2$. The function
\begin{equation}
f:\DD\owns\lambda\mapsto (a_1\lambda^k,\sqrt{1-a_1^2}\lambda^{k+1})
\end{equation}
where $a_1\in(0,1)$  is a $(k+2)$-extremal in the unit ball $\BB_2$ which is not a $(k+2)$-geodesic.
\end{prop}
\begin{proof}
We already know that the analytic disc $f$ is a $(k+2)$-extremal (use Remark~\ref{remark:production-of-extremals-in-classical-domains}). Suppose that there is an $F\in\OO(\BB_2,\DD)$ such that $F(0)=0$ and $F\circ f$ is a non-constant
Blaschke product of degree at most $k+1$. First we consider the case $k\geq 3$.  Expanding around $0$ and comparing the Taylor coefficients on both sides of the equality:
\begin{equation}
F\left(\lambda^ka_1,\lambda^{k+1}\sqrt{1-a_1^2}\right)=B(\lambda)
\end{equation}
we easily get (multiplying by a unimodular constant) that $B(\lambda)=\lambda^k$ or $B(\lambda)=\lambda^km_{\gamma}(\lambda)$ for some $\gamma\in\DD$. Write the Taylor expansion of $F$ in the form $\alpha z_1+\beta z_2+\ldots$. In the first case $|\alpha a_1|=1$ which gives $|\alpha|>1$ -contradiction. Consider now the second case. The coefficient at the left side at $\lambda^{k+2}$ is $0$ which means that $\gamma=0$. Write $F(z_1,z_2)=\alpha z_1+\beta z_2+o(z_1,z_2)$. Since $\limsup_{|\lambda|\to 1}|(F(\lambda(z_1,z_2)))/\lambda|\leq 1$ we easily get that  $\alpha z_1+\beta z_2\in\DD$ for any $z\in\BB_2$ which shows that $|\alpha|^2+|\beta|^2\leq 1$. Then the coefficient at $\lambda^{k+1}$ on the left side is $\beta\sqrt{1-a_1^2}$ which cannot have absolute value one as it does on the right side.

Now we assume that $k=2$.
\begin{equation}
F\left(\lambda^2a_1,\lambda^{3}\sqrt{1-a_1^2}\right)=B(\lambda).
\end{equation}
where $B(\lambda)=\lambda^2m_{\gamma}(\lambda)$ (the case when the degree of $B$ is one or two is simple).
Write the Taylor expansion of $F$ as $\alpha z_1+\beta z_2+\ldots$. Comparing the coefficients at $\lambda^2$ and at $\lambda^3$ leads us to
equalities $|\alpha a_1|=|\gamma|$, $|\beta |\sqrt{1-a_1^2}=1-|\gamma|^2$. We also know that $\alpha z_1+\beta z_2\in\DD$ for any $z\in\BB_2$ which is equivalent to the inequality $|\alpha|^2+|\beta|^2\leq 1$. Consequently,
\begin{equation}
\frac{|\gamma|^2}{a_1^2}+\frac{(1-|\gamma|^2)^2}{1-a_1^2}\leq 1.
\end{equation}
The last inequality is equivalent to
\begin{equation}
|\gamma|^2(1-a_1^2)+(1-|\gamma|^2)^2a_1^2-a_1^2(1-a_1^2)\leq 1
\end{equation}
or
\begin{equation}
(a_1^2-|\gamma|^2)^2+|\gamma|^2(1-|\gamma|^2)(1-a_1^2)\leq 0.
\end{equation}
The last inequality does not hold for any $\gamma\in\DD$, $a_1\in(0,1)$ - contradiction.

\end{proof}


Problem. It would be interesting to know whether 3-extremals in the unit ball are 3-complex geodesic. Even if the answer to this question is positive we think that a counterpart of the Lempert theorem for $3$-extremals in the convex domains does not hold, i.e. there is a convex domain $D$
and a $3$-extremal $f$ in $D$ for which there is no left inverse $F$ such that the composition of $F\circ f$ is a non-constant Blaschke product of degree at most two.

\section{(weak) $m$-extremals in the symmetrised bidisc}

In the present section we get results that show how (weak)  $m$-extremals in the symmetrised bidisc look like. Recall that \textit{the symmetrised bidisc} $\mathbb G_2$ is the image of $\mathbb D^2$ under the mapping $(z_1,z_2)\mapsto(z_1+z_2,z_1z_2)$. In our attitude and important tool will be played by the tetrablock $\mathbb E$ and the possibility of embedding the 
symmetrised bidisc in $\mathbb E$ so we start this section by recalling basic properties of this domain.

\subsection{Geometry of the tetrablock}
Recall that the \textit{ tetrablock } is a domain in $\mathbb C^3$ denoted by $\mathbb E$ and given by the formula
$$\mathbb E=\{(x_1, x_2, x_3)\in \mathbb C^3: |x_1 -\bar x_2 x_3| + |x_2-\bar x_1 x_3| + |x_3|^2<1\}.$$
It is well known that $\mathbb E$ may be given as the image of the Cartan domain of the first type $\mathcal R_I$ under the mapping $\pi:\mathbb C^{2\times 2}\to \mathbb C^3$ $$z=(z_{i,j})\mapsto (z_{11}, z_{22}, \det z).$$
Basic properties of the tetrablock that we use may be found in \cite{Abo 2007}, \cite{AWY}, \cite{Edi-Kos-Zwo}, \cite{Kos 2009}, \cite{You 2007} and \cite{Zwo 2013}.

What is more, $\pi$ restricted to $\mathcal R_{II}$ maps properly $\mathcal R_{II}$ onto $\mathbb E$ and the locus set of this proper mapping $\{(x_1,x_2, x_3)\in \mathbb D^3:\ x_1 x_2=x_3\}$ coincides with the royal variety of $\mathbb E$ which we denote by $\mathcal T$.

In the sequel we shall make use of the structure of the group of automorphisms of $\mathbb E$ and its connection with the group of automorphisms of $\mathcal R_{II}$ (see \cite{Kos 2009}). Recall that the group of automorphisms of $\mathcal R_{II}$ is generated by the linear isomorphisms $$L_U:x\mapsto U x U^t,$$ where $U$ is a unitary matrix, and the mappings 
\begin{equation}\label{eq:autR}
\Phi_a: x\mapsto (1-aa^*)^{-\frac{1}{2}} (x-a)(1-a^* x)^{-1} (1-a^* a)^{\frac{1}{2}},
\end{equation}
where $a\in \mathcal R_{II}$ (see \cite[Section 3]{Har} and \cite{Hua}). Note that $\Phi_a(0)=-a$ and $\Phi_a(a)=0$. Moreover, $\Phi_a^{-1}=\Phi_{-a}$.

It follows from \cite{Kos 2009} (see also \cite{You 2007}) that any automorphism $\psi$ of the tetrablock is of the form $$\psi\circ \pi= \pi \circ \Phi$$ for some automorphism $\Phi$ of $\mathcal R_{II}$. Moreover, such an automorphism $\Phi$ is generated either by $\Phi_A$, where $A=\left(\begin{array}{cc} a_1 & 0\\ 0 & a_2\end{array}\right)$, $a_1,a_2\in \mathbb D$, or by $L_U$, where $U=\left(\begin{array}{cc} \omega_1 & 0\\ 0 & \omega_2\end{array}\right)$ or $U=\left(\begin{array}{cc} 0 & \omega_1 \\ \omega_2 & 0\end{array}\right)$, where $\omega_1,\omega_2\in \mathbb T$.

Some simple computations lead to the description of the group of automorphism of $\mathbb E$ (see also \cite{You 2007}). More precisely, the automorphism $\Phi_A$ of $\mathcal R_{II}$, where $A=\left(\begin{array}{cc} a & 0\\ 0 & b\end{array}\right)$, $a,b\in \mathbb D$, induces the automorphism $\psi$ of $\mathbb E$ given by
\begin{multline}\label{autE}\psi(x_1, x_2, x_3)=\\ \left( \frac{x_1 - a -\bar b x_3+ a \bar b x_2}{1- \bar a x_1 - \bar b x_2 +\bar a \bar b x_3}, \frac{x_2 - b -\bar a x_3+ \bar a b x_1}{1- \bar a x_1 - \bar b x_2 +\bar a \bar b x_3}, \frac{x_3 - a x_2 -b x_1 + a b}{1- \bar a x_1 - \bar b x_2 +\bar a \bar b x_3}\right).\end{multline}

Moreover, $L_U$, where $U$ is a unitary diagonal or anti-diagonal matrix, induces automorphisms of $\mathbb E$ generated by $$(x_1,x_2, x_3)\mapsto (\omega x_1, \eta x_2, \omega \eta x_3),$$ where $\omega, \eta\in \mathbb T$, and by $$(x_1,x_2, x_3)\mapsto (x_2, x_1, x_3).$$

In the sequel we shall make use of the following
\begin{remark}\label{bE} The point $\pi(x)$, where $x\in\partial \mathcal R_I$, lies in the topological boundary of $\mathbb E$ if and only if $|x_{12}|=|x_{21}|$. This is an immediate consequence of Lemma~9 in \cite{Edi-Kos-Zwo} and the properness of $\pi|_{\mathcal R_{II}}:\mathcal R_{II} \to \mathbb E$.

In particular, for $x\in \overline{\mathcal R_I}$ such that $\pi(x)\in \partial \mathbb E$ the following statement is a consequence of the description of the Shilov boundary of the tetrablock (see e.g. \cite[Remark~13]{Kos 2009}):
$$x\in \partial_s \mathcal R_I\quad \text{if and only if}\quad \pi(x)\in \partial_s \mathbb E.$$
\end{remark}

\subsection{(Weak) $m$-extremals in the symmetrised bidisc intersecting $\Sigma$}

At this point we may outline the idea of study of weak $m$-extremals in the symmetrised bidisc. First we concentrate on a more difficult problem of description of $m$-extremals intersecting the royal variety $\Sigma$ (definition is given below).

Since $\pi:\mathcal R_{II}\to \mathbb E$ is proper, $\pi$ restricted to $\mathcal R_{II}\setminus \pi^{-1}(\mathcal T)$ is a holomorphic covering of $\mathbb E\setminus \mathcal T$. Therefore, an analytic disc in $\mathbb E$ omitting its royal variety may be lifted to an analytic disc in the classical Cartan domain of the second type. Of course this does not have to be true if the analytic disc intersects the royal variety of the tetrablock. But then we may lift it to an analytic disc in $\mathcal R_I$ or, up to a composition with an automorphism of the tetrablock, it is of the form $\lambda\mapsto (0, 0, a(\lambda))$ for some $a\in \mathcal O(\mathbb D, \mathbb D)$ (see \cite[Lemma~7]{Edi-Kos-Zwo}).

At this point it would be reasonable to recall the close relationship between the tetrablock and the symmetrised bidisc $\mathbb G_2$. First recall that {\it the symmetrised bidisc} (denoted by $\mathbb G_2$) is a domain which may be defined as the image under the (proper and holomorphic) mapping $(z_1,z_2)\mapsto(z_1+z_2, z_1 z_2)$ of the bidisc $\mathbb D^2$ or equivalently
\begin{equation}
\mathbb G_2:=\{(s,p)\in\mathbb C^2:|s-\overline{s}p|+|p|^2<1\}.
\end{equation}
A similar role to that of $\mathcal T$ (in $\mathbb E$) is played by the royal variety of the symmetrised bidisc, i.e. the set $\Sigma:=\{(2\lambda,\lambda^2):\lambda\in\mathbb D\}$.
We have a natural embedding
$$\iota:\mathbb G_2\ni(s,p)\mapsto (s/2, s/2, p)\in \mathbb E.$$
On the other hand the mapping
$$
p:\mathbb E\ni x\mapsto (x_1+x_2,x_3)\in\mathbb G_2
$$
satisfies the equality $p\circ \iota=\operatorname{id}_{\mathbb G_2}$. We shall use these facts extensively.

\vskip3mm

Consequently, if $f:\mathbb D\to \mathbb G_2$ is an analytic disc, then $(f_1/2, f_1/2, f_2)$ is an analytic disc in the tetrablock $\mathbb E$. Therefore, we may lift it to an analytic disc in $\overline{\mathcal R_I}$.
It is much more comfortable and natural to lift $f$ to an analytic disc in the classical Cartan domain of the second type

To do it let us denote the action permuting columns of a given matrix $a\in \mathbb C^{2\times 2}$ by ${}^\tau a$. These observations lead us to Lemma~\ref{clem} enabling us to transport the problem of the study of ($m$-extremals) analytic discs in the symmetrised bidisc to that in the Cartan domain of the second type. Similar ideas appear also in \cite{Agl}. However, it is quite convenient to pass to the tetrablock. The geometry of this domain and its properties described in \cite{AWY} turn out to be very helpful (e.g. we shall show that properness of weak $m$-extremals in the symmetrised bidisc is a simple consequence of the fact that the tetrablock may be given as the image of the Cartan domain under a proper holomorphic mapping).

\begin{lemma}[see also \cite{Agl}]\label{clem} Let $\varphi:\mathbb D\to \mathbb G_2$ be an analytic disc. Then there is an analytic disc $f:\mathbb D\to \overline{\mathcal R_{II}}$ such that $$\left(\frac{\varphi_1}{2}, \frac{\varphi_1}{2}, \varphi_2\right) = \pi\circ  {}^\tau f.$$
\end{lemma}

\begin{proof}[Proof of Lemma~\ref{clem}]
The result may be deduced from \cite[Theorem~1.3]{Agl}. For the convenience of the Reader we shall present a simple proof involving geometry of the tetrablock. This is also justified by the fact that we shall use more properties and relations between the tetrablock and the Cartan classical domains in the sequel.

To prove the assertion it suffices to observe that $\iota\circ \varphi$ is an analytic disc in $\mathbb E$ and then apply \cite[Lemma~7]{Edi-Kos-Zwo}.
\end{proof}

Of course, the result presented above is the most interesting in the case when $\varphi(\mathbb D)$ intersects $\Sigma$. Note that it is trivial if $\varphi(\mathbb D)$ is contained in $\Sigma$, because then $\varphi=(\varphi_1, \frac{\varphi_1^2}{4}),$ where $\varphi_1\in \mathcal O(\mathbb D, \mathbb D)$.

\begin{remark}\label{LL}
Importance of Lemma~\ref{clem} is a consequence of the following simple observations:
\begin{itemize} \item $\varphi:\mathbb D\to \mathbb G_2$ is a (weak) $m$-extremal in $\mathbb G_2$ if and only if $(\frac{\varphi_1}{2}, \frac{\varphi_1}{2}, \varphi_2)$ is a (weak) $m$-extremal in the tetrablock,
\item if $\pi\circ {}^\tau f :\mathbb D\to \mathbb E$ is a (weak) $m$-extremal, where $f \in \mathcal O(\mathbb D, \mathcal R_{II})$, then $f$ is $m$-extremal in $\mathcal R_{II}$.
\end{itemize}
\end{remark}

Note that a lifting map $f$ appearing in Lemma~\ref{clem} may be chosen in a specific way:
\begin{lemma}\label{cclem} Let $\varphi:\mathbb D\to \mathbb G_2$ be an analytic disc. Then there is an analytic disc $f:\mathbb D\to \overline{\mathcal R_{II}}$ such that $$\left(\frac{\varphi_1}{2}, \frac{\varphi_1}{2}, \varphi_2 \right) = \pi\circ  {}^\tau f$$ and either
\begin{itemize}
\item $f$ lies in the boundary of $\mathcal R_I$ and then $\varphi$ is, up to a composition with an automorphism of the symmetrised bidisc, of the form $(0, \varphi_2)$, or
\item $f$ is an analytic disc in $\mathcal R_{II}$ such that $|f_{11}^*|=|f_{22}^*|$ a.e. on $\mathbb T$ and $f_{11}$ vanishes nowhere on $\mathbb D$.
\end{itemize}

Moreover, if $\varphi$ is a weak $m$-extremal, then $f$ is an $m$-extremal in $\mathcal R_{II}$.
\end{lemma}

To prove the above lemma we need the following observation:
\begin{remark}\label{rem} Suppose that $f:\mathbb D\to \mathbb E$ is a (weak) $m$-extremal in the tetrablock such that $f(0)=0$ and both $f_1$ and $f_2$ are not identically equal to $0$. Suppose additionally that $f=\pi\circ G$, where $G:\mathbb D\to \mathcal R_{I}$ is of the form $G=\left(\begin{array}{cc} f_1 & B_1 g_1\\ B_2 g_2 & f_2\end{array} \right)$, and $B_1$ and $B_2$ are Blaschke products.

Then, of course, the mapping $G$ is an $m$-extremal in $\mathcal R_{I}$. Comparing non-tangential limits we also see that the mapping $H:=\left(\begin{array}{cc} f_1 & g_1\\ B_1 B_2 g_2 & f_2\end{array} \right)$ maps $\mathbb D$ into $\overline{\mathcal R_{I}}$. Note also that $g_1(0)$ lies in the unit disc. Actually, otherwise $g_1$ would be a unimodular constant and therefore $f_1\equiv f_2 \equiv 0$ (as tangential limits of $f_1$ and $f_2$ would vanish on $\mathbb T$). Thus $H$ is an analytic disc in $\mathcal R_{I}$. Since $\pi\circ H=f$, we get that $H=\left(\begin{array}{cc} f_1 & g_1\\ B_1 B_2 g_2 & f_2\end{array} \right)$ is an $m$-extremal in $\mathcal R_{I}$.

\end{remark}

\begin{proof}[Proof of Lemma~\ref{cclem}] The existence of $f$ is guaranteed by Lemma~\ref{clem}.

If $\varphi$ does not intersect $\Sigma,$ then one may assume that ${}^\tau f$ is an analytic disc in $\mathcal R_{II}$ and the assertion is clear. So suppose that $\varphi(\mathbb D)\cap \Sigma\neq \emptyset$. Composing $\varphi$ with an automorphism of $\mathbb G_2$ we may assume that $\varphi(0)=0$.

If $f$ is an analytic disc contained in $\partial \mathcal R_I$, then the assertion follows from Lemma~7 in \cite{Edi-Kos-Zwo}.
In the other case let us denote $g=(g_{ij})={}^\tau f$. Note that if $g_{12}g_{21}\equiv 0$ then we may assume that $g_{12}=g_{21}=0$ and we are done.

If $g_{12}g_{21}\not\equiv 0$, then one may find a Blaschke product such that $g_{12}g_{21}= b h$, where $h:\mathbb D\to \overline{\mathbb D}$ does not vanish. Then $\tilde g$ given by the formula $$\tilde g=\left(\begin{array}{cc} g_{11} & \sqrt{h}\\ b \sqrt{h} & g_{22}\end{array} \right)$$ maps $\DD$ into $\mathcal R_{II}$ and satisfies $\pi\circ \tilde g = \varphi$.

The second part of the lemma is a direct consequence of Remark~\ref{rem}.
\end{proof}

The above lemma suggests that $m$-extremals $\varphi$ such that
\begin{equation}
 \text{(\dag) $\varphi$ is, up to $\Aut(\mathbb G_2)$, of the form $(0, \varphi_2)$}
\end{equation}
should be considered separately.

Using Lemma~\ref{clem} we get the following

\begin{theorem}\label{lem1} Let $\varphi:\mathbb D\to \mathbb G_2$ be a weak $m$-extremal. Assume that $\varphi$ is not of the form (\dag). Then there is $1\leq n\leq m-1$ and there are $a_1,\ldots,a_n\in \mathcal R_{II}$, a holomorphic function $Z:\mathbb D\to \mathbb D$ fixing the origin and unitary matrix $U$ such that  $$\left(\frac{\varphi_1(\lambda)}{2}, \frac{\varphi_1(\lambda)}{2}, \varphi_2(\lambda)\right)= \pi({}^\tau \Phi_{a_1}(\lambda\Phi_{a_2} (\cdots \lambda \Phi_{a_n}( U \left( \begin{array}{cc} \lambda & 0 \\ 0 & Z(\lambda) \end{array}\right) U^t) )), $$ for $\lambda \in \mathbb D.$

We may additionally assume that the lifting disc $$f(\lambda):=\Phi_{a_1}(\lambda\Phi_{a_2} (\cdots \lambda \Phi_{a_n}( U \left( \begin{array}{cc} \lambda & 0 \\ 0 & Z(\lambda) \end{array}\right) U^t) ),\quad \lambda\in \DD$$ is chosen so that $f_{11}$ does not vanish on $\mathbb D$.
\end{theorem}

Before we start the proof of Theorem~\ref{lem1} let us recall the description of complex geodesics in $\mathcal R_{II}$ due to Abate. For the convenience of the Reader we shall present  in the end of the paper a very elementary outline of the proof of this fact.
\begin{lemma}[see \cite{Abate}, Proposition~2.1]\label{ab}
Let $f:\mathbb D\to \mathcal R_{II}$ be a complex geodesic in $\mathcal R_{II}$ such that $f(0)=0$. Then, there are a unitary matrix $U$ and a holomorphic function $Z:\mathbb D\to \mathbb D$ vanishing at the origin such that $$f(\lambda) = U \left( \begin{array}{cc} \lambda & 0 \\ 0 & Z(\lambda) \end{array}\right) U^t ,\quad \lambda \in \DD.$$
\end{lemma}

\begin{proof}[Proof of Theorem~\ref{lem1}]
Let $f:\DD\to \overline{\mathcal R_{II}}$ an analytic disc lifting $\varphi$ to the Cartan domain $\mathcal R_{II}$ whose existence is guaranteed by Lemma~\ref{cclem}. It follows from Remark~\ref{LL} that $f$ is a weak $m$-extremal in $\mathcal R_{II}$, whence it is an $m$-extremal, by Proposition~\ref{prop}. Let $1\leq n\leq m-1$ be the smallest $n$ such that $f$ is an $(n+1)$-extremal. If $n=1$, then $f$ is a complex geodesic and the assertion follows from Lemma~\ref{ab}. Otherwise let us denote $a_1:= -f(0)$. Then $\Phi_{-a_1} \circ f= \Phi_{a_1}^{-1} \circ f$ is an $(n+1)$-extremal such that $\Phi_{-a_1} \circ f(0)=0$. Making use of
Remark~\ref{remark:balanced-extremals} we get an analytic disc $f_1$ in $\mathcal R_{II}$ such that $f(\lambda) = \Phi_{a_1}(\lambda f_1(\lambda))$, $\lambda \in \mathbb D$. Repeating this procedure inductively we get the assertion.
\end{proof}

It seems that $Z$ appearing in Theorem~\ref{lem1} must be rational but we could not prove it. However, we are able to show it for $m=3$:

\begin{theorem}\label{thla} Let $\varphi:\mathbb D\to \mathbb G_2$ be a weak $3$-extremal with respect to $0, \sigma_1,\sigma_2\in \mathbb D$, intersecting $\Sigma$.
Then, either
\begin{enumerate}
\item $\varphi$ is up to a composition with an automorphism of $\mathbb G_2$ of the form $(0, \varphi_2)$, where $\varphi_2$ is a Blaschke product of degree $2$, or
\item $\varphi$ lies entirely in $\Sigma$, i.e. $\varphi=\left(\varphi_1,\frac{\varphi_1^2}{4}\right)$, where $\varphi_1$ is a Blaschke product of degree at most $2$, or
\item there are $a_1,a_2 \in \mathcal R_{II}$, and a unitary symmetric matrix $U$ such that  $$\left(\frac{\varphi_1(\lambda)}{2}, \frac{\varphi_1(\lambda)}{2}, \varphi_2(\lambda) \right)= \pi({}^\tau \Phi_{a_1}(\lambda \Phi_{a_2}(U \lambda)) ),\quad \lambda \in \mathbb D,$$
\item there is $a\in \mathcal R_{II}$, a unitary matrix $U$ and a M\"obius function $m$ such that  $$\left(\frac{\varphi_1(\lambda)}{2}, \frac{\varphi_1(\lambda)}{2}, \varphi_2(\lambda) \right)= \pi({}^\tau \Phi_{a}(U \left(\begin{array}{cc} \lambda& 0 \\ 0 & \lambda m(\lambda)\end{array}\right) U^t) ),\quad \lambda \in \mathbb D.$$
\end{enumerate}
\end{theorem}

To prove Theorem~\ref{thla} we need the following preparatory result.
\begin{lemma}\label{lemprop}
Any $m$-extremal in the tetrablock or in the symmetrised bidisc is proper.
\end{lemma}

\begin{proof}[Proof of Lemma~\ref{lemprop}]
Since the embedding of $\mathbb G_2$ into $\mathbb E$ is proper it suffices to show the assertion for the tetrablock.
First recall that any $m$-extremal in $\mathcal R_{II}$ or in $\mathcal R_I$ is proper (see Proposition~\ref{prop:proper-extremals}).

Recall that $\pi(x)$ lies in the topological boundary of $\mathbb E$, where $x\in \partial \mathcal R_I$, if and only if $|x_{12}|=|x_{21}|$ (see Remark~\ref{bE}).

Now assume that $f:\mathbb D\to \mathbb E$ is a weak $m$-extremal. Then, using Lemma~\ref{cclem} we may lift it to an $m$-extremal $g$ in $\mathcal R_I$ such that $|g_{12}^*|=|g_{21}^*|$ almost everywhere on $\mathbb T$. Thus, to get the assertion it suffices to make use of the fact that $g$ is proper.
\end{proof}

\begin{proof}[Proof of Theorem~\ref{thla}] Composing $\varphi$ with an automorphism of $\mathbb G_2$ we may assume that $\varphi(\sigma_0)=0$ for some $\sigma_0\in \mathbb D$. If $\varphi(\mathbb D)\subset \Sigma$ or $\varphi$ is of the form (\dag) the assertion is clear.

Since the automorphism $\Phi_{\Lambda}$, where $\Lambda=\left( \begin{array}{cc} 0 & 0 \\ 0 & \alpha \end{array}\right)$ maps $\left( \begin{array}{cc} \lambda & 0 \\ 0 & \lambda \end{array}\right)$ to $\left( \begin{array}{cc} \lambda & 0 \\ 0 & m_\alpha(\lambda) \end{array}\right)$, it suffices to show that there are an $a\in \mathcal R_{II}$ and $\Phi\in \Aut(\mathcal R_{II})$ and a M\"obius map $m$ such that
$$(\varphi_1(\lambda)/2, \varphi_1(\lambda)/2, \varphi_2(\lambda))= \pi({}^\tau \Phi_{a}(\lambda \Phi(\left( \begin{array}{cc} \lambda &0 \\ 0 & m(\lambda) \end{array} \right) ) )$$ or $$\left(\frac{\varphi_1(\lambda)}{2}, \frac{\varphi_1(\lambda)}{2}, \varphi_2(\lambda) \right)= \pi({}^\tau \Phi(\lambda) ),\quad \lambda \in \mathbb D,$$

Lifting $\varphi$ as in Theorem~\ref{lem1} we get a 3-extremal $F:\mathbb D\to \mathcal R_{II}$ such that $\pi\circ {}^\tau F = \varphi$, where
\begin{equation}\label{F}
F(\lambda)= \Phi_a\left(\lambda \Phi\left( \left(\begin{array}{cc} \lambda & 0\\ 0 & Z(\lambda)
\end{array}\right)\right)\right),\quad \lambda \in \mathbb D.
\end{equation} 
or 
\begin{equation}\label{Fnew}
F(\lambda)=\Phi\left( \left(\begin{array}{cc} \lambda & 0\\ 0 & Z(\lambda)
\end{array}\right)\right),\quad \lambda \in \mathbb D,
\end{equation}
for some $Z\in \mathcal O(\mathbb D, \mathbb D)$, $\Phi\in \Aut(\mathcal R_{II})$ and $a\in \mathcal R_{II}$. First observe that the choice of $F$ implies, in particular, that the functions $|F_{11}|$ and $|F_{22}|$ are different. Actually, otherwise $F_{11} = \omega F_{22}$ for some $\omega \in \TT$. Since $F_{11}$ vanishes nowhere on $\DD$ we find that $\varphi$ does not intersect $\Sigma$, which gives a contradiction.

Note that our aim is to show that $Z$ is a M\"obius map in the first case or that $Z$ is a Blaschke product of degree at most two in the second one. Seeking a contradiction suppose that it is not the case. Then there is a holomorphic function $\tilde Z$ defined on a neighborhood of $\bar{\mathbb D}$ such that $\tilde Z(\sigma_i)=Z(\sigma_i)$, $i=1,2$ (and additionally $\tilde Z(0)=Z(0)$ in the second case), which is not a Nash function. Recall that a holomorphic function $f$ on $\DD$ is called a \emph{Nash function} if there is a non-zero complex polynomial $P:\CC^n \times \CC\to \CC$ such that $P(\lambda, f(\lambda))=0$ for $\lambda\in \mathbb D$. For basic properties of Nash functions we refer the Reader to \cite{Two}. The important fact that will be used in the sequel of the proof is that the set of Nash functions forms a subring of the ring of holomorphic functions on the unit disc (see \cite[Corollary~1.11]{Two}).

Putting $$\tilde F(\lambda)=  \Phi_a\left(\lambda \Phi\left(\begin{array}{cc} \lambda & 0\\ 0 & \tilde Z(\lambda)
\end{array}\right)\right),\quad \lambda \in \mathbb D,$$
(or $$\tilde F(\lambda)=  \Phi\left(\begin{array}{cc} \lambda & 0\\ 0 & \tilde Z(\lambda)
\end{array}\right),\quad \lambda \in \mathbb D),$$
we get a $3$-extremal in $\mathcal R_{II}$ such that $\pi\circ {}^\tau \tilde F$ is a $3$-extremal in $\mathbb E$. Note that $F_{11}F_{22}\not\equiv 0$ on $\mathbb D.$

Using Lemma~\ref{lemprop} we find that $|\tilde F_{11}|=|\tilde F_{22}|$ on $\mathbb T$. Thus there are finite Blaschke products or unimodular constants $B_1$, $B_2$ and a holomorphic function $g$ such that $\tilde F_{11}=B_1 g$ and $\tilde F_{22}=B_2 g$. Put $f:=\tilde F_{12}=\tilde F_{21}$. Then 
$$  \Phi_a\left(\lambda \Phi\left(\begin{array}{cc} \lambda & 0\\ 0 & \tilde Z(\lambda)
\end{array}\right)\right) = \left(\begin{array}{cc} B_1(\lambda) g(\lambda) & f(\lambda)\\ f(\lambda) & B_2(\lambda) g(\lambda)\end{array} \right)$$
or
$$   \Phi\left(\begin{array}{cc} \lambda & 0\\ 0 & \tilde Z(\lambda)
\end{array}\right) = \left(\begin{array}{cc} B_1(\lambda) g(\lambda) & f(\lambda)\\ f(\lambda) & B_2(\lambda) g(\lambda)\end{array} \right).$$ We will modify the above relation so that $B_1=1$. Then some technical arguments will provide us with a contradiction.

First note that at least one of $B_1$ or $B_2$ is not a unimodular constant. Indeed, if it were not true, then one may use arguments involving the concept of Nash functions. More precisely, we proceed as follows. Put
\begin{equation}\label{G} \tilde F(\lambda, z):= \Phi_a(\lambda \Phi\left(\begin{array}{cc} \lambda & 0\\ 0 & z
\end{array}\right)),\quad \lambda,z \in \mathbb D
\end{equation}
(or analogously 
\begin{equation} \tilde F(\lambda, z):=  \Phi\left(\begin{array}{cc} \lambda & 0\\ 0 & z
\end{array}\right),\quad \lambda,z \in \mathbb D.)
\end{equation}
Since $B_1,B_2\in \mathbb T$, we find that  $\tilde F_{11}(\lambda, \tilde Z(\lambda)) = \omega \tilde F_{22}(\lambda, \tilde Z(\lambda))$ on $\mathbb D$ where $\omega$ is a unimodular constant $\omega=B_1/B_2$. Since $\tilde Z$ is not Nash, the equality $\tilde F_{11}(\lambda,z)=\omega \tilde F_{22}(\lambda, z)$ holds for all $(\lambda, z)\in \mathbb D^2$. In particular, $\tilde F_{11}(\lambda, Z(\lambda))= \omega \tilde F_{22}(\lambda, Z(\lambda))$, $\lambda\in \mathbb D$, which gives a contradiction (as already mentioned, $|F_{11}|\not\equiv|F_{22}|$).

Note that
\begin{equation}\label{cont}\lambda \mapsto \left(\begin{array}{cc} g(\lambda) & f(\lambda)\\ f(\lambda) & B(\lambda) g(\lambda)\end{array} \right)\end{equation}
is a $3$-extremal mapping in $\mathcal R_{II}$, where $B=B_1 B_2$. Actually, this follows from Remark~\ref{rem}. Now we are in a position which gives us a contradiction.

Composing \eqref{cont} with a M\"obius map we may assume that $B$ vanishes at the origin. Moreover, after a composition with $\Phi_{\tilde \alpha}$, where $\tilde \alpha=\left(\begin{array}{cc} 0 & f(0)\\ f(0) & 0\end{array}\right)$, replacing $f$ and $g$ with other we may assume that $f(0)=0$.

Therefore, thanks to the procedure described in Remark~\ref{remark:balanced-extremals} there are $Z\in \mathcal O(\mathbb D, \mathbb D)$, and an automorphism $\Phi$ of $\mathcal R_{II}$ such that
\begin{equation}\label{cont1}\left(\begin{array}{cc} g(\lambda) & f(\lambda)\\ f(\lambda) & B(\lambda) g(\lambda)\end{array} \right)= \Phi_C(\lambda \Phi\left(\begin{array}{cc} \lambda & 0\\ 0 & Z(\lambda)
\end{array}\right)),\quad \lambda\in \mathbb D,\end{equation}
or 
\begin{equation}\label{cont2}\left(\begin{array}{cc} g(\lambda) & f(\lambda)\\ f(\lambda) & B(\lambda) g(\lambda)\end{array} \right)= \Phi\left(\begin{array}{cc} \lambda & 0\\ 0 & Z(\lambda)
\end{array}\right),\quad \lambda\in \mathbb D,\end{equation}
where $C=\left( \begin{array}{cc} c & 0 \\ 0 & 0 \end{array}\right)$ and $c=g(0)$. Since the set of Nash functions on a domain of $\mathbb C^n$ forms a subring of the ring of holomorphic functions  we easily find that $Z$ is not a Nash function. The goal is to derive a contradiction directly from~\eqref{cont1}. To do it we will find an explicit formula for $\Phi$ and then, to simplify the computations, we will pass to the tetrablock.

Assume first that \eqref{cont1} holds. Put
\begin{equation}\label{G1} G(\lambda, z):= \Phi_C\left(\lambda \Phi\left(\begin{array}{cc} \lambda & 0\\ 0 & z
\end{array}\right)\right),\quad \lambda,z \in \mathbb D.
\end{equation}

Since $G_{22}(\lambda, Z(\lambda)) = B(\lambda) G_{11} (\lambda, Z(\lambda)),$ $\lambda \in \mathbb D$, making use once again of the fact that automorphisms of the classical domain are rational and $Z$ is not a Nash function we find that $$G_{22}(\lambda, z) = B(\lambda) G_{11}(\lambda, z), \quad \lambda, z \in \mathbb D.$$ By \eqref{G}, $G_{11}(0,\cdot)\equiv -c$ and $G_{12}(0,\cdot)\equiv G_{21}(0,\cdot)\equiv 0$. Moreover, $G_{12}\equiv G_{21}$.

Write $G_{11}(\lambda, z) = -c+ \lambda g_{11}(\lambda, z)$, $G_{12}(\lambda, z) = \lambda g_{12}(\lambda, z)$, $G_{21}(\lambda, z) = \lambda g_{21}(\lambda, z)$, $\lambda, z \in \mathbb D$. Additionally, put $g_{12}:= g_{21}$.

The explicit formula for $\Phi_C^{-1}$ is $$\Phi_C^{-1}(x)=\left( \begin{array}{cc}\frac{x_{11}+c}{1+\overline c x_{11}}  & \sqrt{1-|c|^2} \frac{x_{12}}{1+ \overline c x_{11}} \\ \sqrt{1-|c|^2}\frac{x_{21}}{1+\overline c x_{11}} &  \frac{x_{22}+\overline c\det x}{1+\overline cx_{11}} \\ \end{array} \right),\quad x=(x_{ij})\in \mathcal R_I.$$ Write $B(\lambda)= \lambda b(\lambda)$. Composing the relation~\eqref{G} with $\Phi_C^{-1}$, dividing by $\lambda$ and putting $\lambda=0$ we infer that
$$\Phi\left(\begin{array}{cc} 0 & 0 \\ 0 & z\end{array} \right) = \left( \begin{array}{cc} \frac{g_{11}(0)}{1-|c|^2} & \frac{g_{12}(0)}{\sqrt{1-|c|^2}} \\ \frac{g_{21}(0)}{\sqrt{1-|c|^2}} & - b(0) c\end{array}\right),\quad z\in \mathbb D. $$

The first observation is that the component $(\Phi\left(\begin{array}{cc} 0 & 0 \\ 0 & z\end{array} \right))_{22}$ is a constant not depending on $z$. Therefore the composition $\Phi$ with $\Phi_D$, where $D=\left(\begin{array}{cc} 0 & 0\\ 0 & d \end{array}\right)$ and $d=-b(0)c$ has  the property $(\Phi_D\circ \Phi\left(\begin{array}{cc} 0 & 0 \\ 0 & z\end{array} \right))_{22}= 0$. This means that the the function $a:=(\Phi\left(\begin{array}{cc} 0 & 0 \\ 0 & \cdot \end{array} \right))_{12}$ satisfies
\begin{equation}\label{lhs}\det(\Phi_D\circ \Phi\left(\begin{array}{cc} 0 & 0 \\ 0 & z\end{array} \right)) = a^2(z),\quad z\in \mathbb D.\end{equation}
Thanks to the form of automorphisms of $\mathcal R_{II}$ we see that the function in the left side in the equation \eqref{lhs} is either constant of it is a rational function of degree $1$. Hence, $a$ is constant.

Therefore $$\Phi_D\circ \Phi\left(\begin{array}{cc} 0 & 0 \\ 0 & \cdot\end{array} \right)= \left(\begin{array}{cc} \alpha(\cdot) & a \\ a & 0\end{array} \right)$$ for some holomorphic function $\alpha$. Composing the above relation with $\Phi_A$, where $A=\left(\begin{array}{cc} 0 & a\\ a & 0 \end{array}\right)$ we find that
\begin{equation}\label{dt} \Phi_A \circ \Phi_D\circ \Phi\left(\begin{array}{cc} 0 & 0 \\ 0 & \cdot\end{array} \right)= \left(\begin{array}{cc} \frac{\alpha(\cdot)}{1-|a|^2} & 0 \\ 0 & 0\end{array} \right).
\end{equation}
Put $\Psi=\Phi_A\circ  \Phi_B\circ \Phi$. Using the description of $\Aut(\mathcal R_{II})$ we see that $\Psi$ is of the form $\Psi= U \Phi_\Gamma U^t$, where $\Gamma=\left( \begin{array}{cc} \gamma_1& \gamma_0\\  \gamma_0 & \gamma_2\end{array}\right)$ and $U$ is unitary. Comparing the determinants in \eqref{dt} we see that $$\frac{-\gamma_1 (z- \gamma_2) - \gamma_0^2}{1- \bar \gamma_1 z} = 0\quad \text{for any}\ z \in \mathbb D,$$ whence $\gamma_1= \gamma_0=0$. Simple calculations show that $U$ is anti-diagonal. To simplify the notation let us assume that $U={}^\tau 1$ (the argument is the same in the general case).

Summing up we have obtained the relation $$\Phi=\Phi_{-D} \circ \Phi_{-A} \circ (U\Phi_{\Gamma} U^t),$$ whence
$$\Phi_C\left(\lambda \Phi_{-D}\left(\Phi_{-A}\left(U \Phi_\Gamma\left(\begin{array}{cc} \lambda & 0\\ 0 & z\end{array} \right) U^t\right)\right)\right) = \left( \begin{array}{cc} G_{11}(\lambda, z) & G_{12}(\lambda, z) \\ G_{21}(\lambda, z) & \lambda b(\lambda) G_{11}(\lambda, z)\end{array} \right)$$ for $\lambda, z\in \mathbb D.$
In particular, replacing $z$ with $m_{\gamma_2}(z)$, we get that the equality
\begin{multline}\label{eqk} \left(\Phi_C\left(\lambda \Phi_{-D}\left(\Phi_{-A}\left(\begin{array}{cc}  z& 0\\ 0 &
\lambda\end{array} \right)\right)\right)\right)_{22}=\\ \lambda b(\lambda) \Phi_C\left(\lambda \Phi_{-D}\left(\Phi_{-A}\left(\begin{array}{cc} z & 0\\ 0 & \lambda \end{array} \right)\right)\right)_{11}
\end{multline}  holds for any $\lambda$ and $z$ (whenever it is well defined).

Here it is convenient to pass to $\mathbb E$ (as it lies in 3-dimensional space). The equation~\eqref{autE} gives the following formula for the automorphism induced by $\Phi_{C}$:
$$\psi_1(z) = \left( \frac{z_1 - c}{ 1- \bar c z_1}, \frac{z_2 - \bar c z_3}{1- \bar c z_1} , \frac{ z_3 - c z_2}{1- \bar c z_1}\right)$$
and the one induced by $\Phi_{-D}$:
$$\psi_2(z) = \left( \frac{z_1 + \bar d z_3}{1+ \bar d z_2}, \frac{z_2 + d}{1+ \bar d z_2}, \frac{z_3 + dz_1}{1+ \bar d z_2} \right).$$
Then \eqref{eqk} may be rewritten in the following way:
\begin{align*}
\left(\psi_1(\lambda. \psi_2 (\frac{z (1- |a|^2)}{1-\bar a^2 \lambda z}, \frac{\lambda  (1- |a|^2)}{1-\bar a^2 \lambda z}, \frac{(1- |a|^2)^2\lambda z - (\bar a \lambda z - a )^2}{(1-\bar a^2 \lambda z)^2})
\right)_2=\\
B(\lambda)\left(\psi_1(\lambda. \psi_2 (\frac{z(1- |a|^2)}{1-\bar a^2 \lambda z}, \frac{\lambda  (1- |a|^2)}{1-\bar a^2 \lambda z}, \frac{(1- |a|^2)^2\lambda z - (\bar a \lambda z - a )^2}{(1-\bar a^2 \lambda z)^2})
\right)_1,
\end{align*}
where $\lambda\mapsto \lambda. x$ is an action on $\mathbb C^3$ given by $\lambda. x = (\lambda x_1, \lambda x_2, \lambda^2 x_3)$, $x\in \mathbb C^3$, $\lambda \in \mathbb C$.

Thus we have the equality
$$\frac{y_2 - \bar c \lambda y_3}{\lambda y_1 - c} = b(\lambda),$$
where $y=\psi_2 (\frac{z (1- |a|^2)}{1-\bar a^2 \lambda z}, \frac{\lambda (1- |a|^2)}{1-\bar a^2 \lambda z}, \frac{(1- |a|^2)^2\lambda z - (\bar a \lambda z - a )^2}{(1-\bar a^2 \lambda z)^2})$.
Consequently
\begin{equation}\label{end}\frac{x_2 + d - \lambda \bar c x_3 - \lambda \bar c d x_1}{-c - c \bar d x_2 + \lambda x_1 + \lambda \bar d x_3}= b(\lambda),
\end{equation}
where $x=(\frac{z (1- |a|^2)}{1-\bar a^2 \lambda z}, \frac{\lambda (1- |a|^2)}{1-\bar a^2 \lambda z}, \frac{(1- |a|^2)^2\lambda z - (\bar a \lambda z - a )^2}{(1-\bar a^2 \lambda z)^2})$.

Recall that $-c b(0)= d$ (or put $\lambda=0$ in \eqref{end}).

Assume first that $a\neq 0$. Then letting $z\to \infty$ we find that
\begin{equation}\label{end1}\frac{x_2 + d - \lambda \bar c x_3 - \lambda \bar c d x_1}{-c - c \bar d x_2 + \lambda x_1 + \lambda \bar d x_3}= b(\lambda)\end{equation}
holds for $x=(\frac{1-|a|^2}{-\bar a^2 \lambda}, 0, \frac{-1}{\bar a^2})$. Putting this $x$ into \eqref{end1}, we get
$$\frac{d + \lambda \bar c \frac{1}{\bar a^2} + \bar c d \frac{1-|a|^2}{\bar a ^2}}{-c + \frac{1-|a|^2}{\bar a^2} - \lambda \frac{\bar d}{\bar a^2}} = b(\lambda).$$ Putting $\lambda = 0$ we find that $$\frac{1+ \bar c \frac{1-|a|^2}{\bar a^2}}{c+ \frac{1-|a|^2}{\bar a^2}} = \frac{1}{c},$$ whence either $|a|=1|$ or $|c|=1.$

If $a=0$ the situation is simpler. Indeed, then $x=(z, \lambda, \lambda z)$. Taking $z=0$ and putting it into \eqref{end} we get that
$$\frac{\lambda + d}{-c - c\bar d \lambda} = b(\lambda).$$ Since $b$ is a Blaschke product, $|c|=1$; a contradiction.

We are left with the case \eqref{cont2}. Recall that $f(0)=0$ and $B(0)=0$.
 Then $\Phi(x)=\Phi_D(U x U^t)$, where $C=\left(\begin{array}{cc} c & 0\\ 0 & 0
 \end{array}
 \right),$ $c=g(0)$ and $U$ is unitary. Then we easily get the following relation:
 $$\frac{u_{21}^2\lambda +  u_{22}^2 Z(\lambda) + \bar c (\det U)^2 \lambda Z(\lambda)}{u_{11}^2 \lambda + u_{12}^2 Z(\lambda) + c}=B(\lambda).
 $$
Since $Z$ is not a Nash function we easily deduce that the equality
$$\frac{u_{21}^2\lambda +  u_{22}^2 z + \bar c (\det U)^2 \lambda z}{u_{11}^2 \lambda + u_{12}^2 z + c}=B(\lambda)
 $$ holds for any $\lambda, z\in \DD$. From this we immediately derive a contradiction (put $\lambda=0$ and then $z=-c u_{12}^{-2}$).
\end{proof}

\subsection{(Weak) $3$-extremals in the symmetrised bidisc omitting $\Sigma$}

In the case when weak 3-extremals do not touch the royal variety we get the following
\begin{theorem}\label{thlb} Let $f:\mathbb D\to \mathbb G_2$ be a weak $3$-extremal in $\mathbb G_2$ such that $f(\mathbb D)\cap \Sigma = \emptyset$. Then there are Blaschke products $B_1,B_2$ of degree at most $2$ such that $$f=(B_1 + B_2, B_1 B_2).$$
\end{theorem}

\begin{proof}[Proof of Theorem~\ref{thlb}]
If $f:\mathbb D\to \mathbb G_2$ is a weak $3$-extremal such that $f(\mathbb D)\cap \Sigma = \emptyset$, then we may lift it to $\mathbb D^2$. Namely, there are $\varphi_1,\varphi_2\in \mathcal O(\mathbb D, \mathbb D)$ such that $$f=(\varphi_1+\varphi_2, \varphi_1 \varphi_2).$$ Since $f$ is a weak $3$-extremal we get that $\varphi=(\varphi_1, \varphi_2)$ is $3$-extremal in the bidisc. Therefore one of functions $\varphi_1$ or $\varphi_2$ is a Blaschke product of degree at most $2$. Losing no generality assume that it is $\varphi_1$.

We claim that $\varphi_2$ is a Blaschke product of degree at most $2$, too. Otherwise, one can find a non-rational holomorphic function $\psi:\overline{\mathbb D}\to \mathbb D$ which agrees with $\varphi_2$ at $3$ given points and such that $\psi(\mathbb D)\subset \subset \mathbb D$. Then $(\varphi_1+\psi, \varphi_1 \psi)$ is a weak $3$-extremal in $\mathbb G_2$. It follows form Rouch\'e theorem that the equation $\varphi_1=\psi$ has at least one solution in $\mathbb D$, whence $(\varphi_1+\psi, \varphi_1 \psi)$ is an irrational weak $3$-extremal in $\mathbb G_2$ intersecting the royal variety $\Sigma$. This contradicts Theorem~\ref{thla}.
\end{proof}

\subsection{3-extremals in the symmetrised bidisc are inner and rational of degree at most $4$}

As a consequence of Theorems~\ref{thla} and \ref{thlb} we obtain an affirmative answer to a question posed in \cite{Agl-Lyk-You2 2013}.
\begin{theorem}\label{thm:three-extremals-are-rational}
Let $\varphi:\mathbb D\to \mathbb G_2$ be a $3$-extremal mapping. Then:
\begin{itemize}
\item $\varphi$ is a rational function of degree at most $4$;
\item $\varphi$ is $\mathbb G_2$-inner, i.e. $\varphi(\mathbb T)\subset \partial_s\mathbb G_2,$ where $\partial_s \mathbb G_2=\{(z+w,zw): |z|=|w|=1\}$ is the Shilov boundary of $\mathbb G_2$.
\end{itemize}
\end{theorem}

\begin{proof}[Proof of Theorem~\ref{thm:three-extremals-are-rational}] If $\varphi$ omits $\Sigma$, the assertion is a direct consequence of Theorem~\ref{thlb}. Similarly, the assertion is clear if $\varphi$ lies entirely in $\Sigma$.

Suppose that $\varphi(\mathbb D)$ touches $\Sigma$ and $\varphi(\mathbb D)\not\subset \Sigma$.

If $\left(\frac{\varphi_1(\lambda)}{2}, \frac{\varphi_1(\lambda)}{2}, \varphi_2(\lambda) \right)= \pi({}^\tau \Phi_{a_1}(U \left(\begin{array}{cc}  \lambda & 0 \\ 0 & \lambda m(\lambda ) \end{array} \right) U^t) ),$ $\lambda \in \mathbb D,$ then the assertion is clear. So suppose that such a representation is not possible. Let $\sigma_0$ be such that $\varphi(\sigma_0)\in \Sigma$. Since a composition with a M\"obius map does not change a degree of a rational mapping we may assume that $\sigma_0=0.$ Then, there are $a\in \mathcal R_{II}$, and $\Phi\in \Aut(\mathcal R_{II})$ such that  $$\left(\frac{\varphi_1(\lambda)}{2}, \frac{\varphi_1(\lambda)}{2}, \varphi_2(\lambda) \right)= \pi({}^\tau \Phi_{a}(\lambda \Phi(\lambda)) ),\quad \lambda \in \mathbb D.$$

Note that if $\Phi(0)=0$, then all but one entries of the matrix $a$ are equal to $0$ (the element lying on the diagonal may not vanish). Then a straightforward calculation shows that $\Phi_{ij}$ are of the form $\frac{p_{ij}}{q},$ where $p_{ij}$, $q$ are polynomials of degree at most $4$.

Therefore, in a general case, that is when $\varphi(0)\in \Sigma$, it suffices to compose $\varphi$ with an automorphism of $\mathbb E$.

To prove the second part note that $\Phi_a(\lambda \Phi(\lambda))\in \partial_s \mathcal R_{II}$ for $\lambda\in \TT$, whence ${}^\tau \Phi_a(\lambda \Phi(\lambda))\in \partial_s \mathcal R_{I}$ for any $\lambda \in \TT$. Since $\varphi$ is proper, we see that $\pi({}^\tau \Phi_a(\lambda \Phi(\lambda))) \in \partial \EE$ for $\lambda \in \TT$. Thus, the assertion of this part follows from the fact that $\pi(x)$ lies in the Shilov boundary of $\mathbb E$ for any $x$ lying in the Shilov boundary of $\mathcal R_I$ and such that $\pi(x)\in \EE$. (Remark~\ref{bE}).
\end{proof}

\section{Identity is $m$-extremal}\label{Ident}
We conclude the paper with a simple observation on a more general notion of $m$-extremals introduced in \cite{Agl-Lyk-You2 2013} and we present a solution of a problem posed there.

Similarly to the case of mappings defined on the unit disc we introduce the notion of $m$-extremal mappings defined on a general domain.

Let $D$ be a bounded domain in $\CC^N$ and $\Omega$ a domain in $\CC^M$. Let $m\geq 2$ and let $\lambda_1,\ldots,\lambda_m\in D$ be pairwise different points
and $z_1,\ldots,z_m\in\Omega$. Following \cite{Agl-Lyk-You 2013}
we say that the interpolation data
\begin{equation}
 \lambda_j\mapsto z_j,\ D\to\Omega,\quad j=1,\ldots,m
\end{equation}
is {\it extremally solvable} if there is a map $h\in\OO(D,\Omega)$ such that $h(\lambda_j)=z_j$, $j=1,\ldots,m$
and for any open neighborhood $U$ of $\bar D$ there is no $f\in\OO(U,\Omega)$ such that $f(\lambda_j)=z_j$, $j=1,\ldots,m$.

We say that $h\in\OO(D,\Omega)$ is {\it $m$-extremal} if for all choices of $m$ pairwise distinct points $\lambda_1,\ldots,\lambda_m\in D$
the interpolation data
\begin{equation}
 \lambda_j\mapsto h(\lambda_j),\ D\to\Omega,\quad j=1,\ldots,m
\end{equation}
is extremally solvable. Note that if $h$ is $m$-extremal then it is $m+1$ extremal.

\bigskip

The question posed (and partially solved) in \cite{Agl-Lyk-You 2013} (Proposition 2.5 and remark just before it) is whether the identity is $m$-extremal. The answer is yes.

Before we present the proof let us recall definition of the Lempert function of a domain $D$ of $\CC^n$. It is a holomorphically invariant function denoted by $\tilde k_D$ and given by the formula (for basic properties of the Lempert function see \cite{Jar-Pfl 1993}):
$$\tilde k_D(z,w) = \inf \rho(0,\sigma), $$ where the infimimum is taken over all $\sigma>0$ such that there is an analytic disc $f:\DD\to D$, $f(0)=z$ and $f(\sigma)=w$. Here $\rho$ is the Poincar\'e distance on $\DD$ given by the formula
\begin{equation}
\rho(\lambda_1,\lambda_2):=\frac{1}{2}\log\frac{1+\left|\frac{\lambda_1-\lambda_2}{1-\bar\lambda_1\lambda_2}\right|}{1-\left|\frac{\lambda_1-\lambda_2}{1-\bar\lambda_1\lambda_2}\right|},\;\lambda_1,\lambda_2\in\DD.
\end{equation}
Note that we changed the letter denoting the Poincar\'e distance to avoid the ambiguities that could be caused by the earlier use of the letter $p$ in the other context.
\begin{remark}\label{rem:fin}
Recall that the analytic disc passing through $z$ and $w$ lying in $D$ required in the definition of the Lempert function always exists (see e.g. Remark 3.1.1 in \cite{Jar-Pfl 1993}). Note also that applying the definition of the Lempert function together with the standard reasoning employing the Montel theorem we get for any bounded domain $D$ in $\mathbb C^n$ and for any $w,z\in D$ that there is a mapping $f:\DD\to \overline{D}$ such that
$f(0)=w$, $f(\sigma)=z$ and $p(0,\sigma)=\tilde k_D(w,z)$. Note that under the additional assumption that $D$ is \textit{taut} (i.e. the family $\mathcal O(\DD,D)$ is normal) we may assume that $f(\DD)\subset D$ -- see Proposition 3.2.4 in \cite{Jar-Pfl 1993}.
\end{remark} Keeping in mind these observations we are able to prove the following result.

\begin{prop}\label{prop:identity-is-extremal}
 Let $D$ be a bounded domain in $\CC^N$. Then the identity mapping $\id_D:D\to D$ is $m$-extremal for any $m\geq 2$.
\end{prop}
\begin{proof}[Proof]
It is sufficient to show that $\id_D$ is $2$-extremal. Suppose that it does not hold. Then there are points $w,z\in D$, $w\neq z$, an open $U\supset\bar D$ and $h\in\OO(U,D)$ such that $h(w)=w$, $h(z)=z$. Making use of Remark~\ref{rem:fin} (i.e. making use of the definition of the Lempert function and applying standard reasoning employing the Montel theorem) we get a $\sigma\in (0,1)$ such that there is a mapping $f:\DD\to \overline{D}$ with
$f(0)=w$, $f(\sigma)=z$ and $\rho(0,\sigma)=\tilde k_D(w,z)$. Now the function $g:=h\circ f\in\OO(\DD,D)$ satisfies the following properties $g(0)=w$,
$g(\sigma)=z$ and $g(\DD)\subset h(\overline{D})\subset\subset D$. For $0<t<1$ we define $\tilde g(\lambda):=g(\lambda)+\frac{\lambda}{t\sigma}\left(g(\sigma)-g(t\sigma)\right)$, $\lambda\in\mathbb D$. Note that for $0<t<1$ sufficiently close to $1$ we get that $\tilde g\in\OO(\mathbb D,D)$, $\tilde g(0)=g(0)=w$, $\tilde g(t\sigma)=g(\sigma)=z$ so $\tilde k_D(w,z)\leq \rho(0,t\sigma)<\rho(0,\sigma)$ -- a contradiction.
\end{proof}

\section{Appendix. Sketch of proof of Lemma~\ref{ab}}
\begin{proof}[Sketch of proof of Lemma~\ref{ab}]
Fix $\sigma\in (0,1)$ and let $a:=f(\sigma)\in \mathcal R_{II}\setminus \{0\}$. Applying the singular value decomposition we find that there is a unitary matrix $U$ and $0\leq \lambda_2\leq \lambda_1< 1$ such that $a= U \left( \begin{array}{cc} \lambda_1 & 0 \\ 0 & \lambda_2 \end{array} \right) U^t$. Since $L_U: x\mapsto U x U^t$ is an automorphism of $\mathcal R_{II}$, composing $f$ with $L_U^{-1}$ we may assume that $U=1$.


Let $g:\DD\to \DD$ be any holomorphic mapping such that $g(0)=0$ and $g(\lambda_1) = \lambda_2$. Clearly, $$\lambda\mapsto \left( \begin{array}{cc} \lambda & 0 \\ 0 & g(\lambda) \end{array} \right)$$ is a complex geodesic in $\mathcal R_{II}$, as $\mathcal R_{II}\ni z \mapsto z_{11}$ is its left inverse, whence $\tilde k_{\mathcal R_{II}}(a, 0)= \rho(\lambda_1, 0)$. Since $f$ is a complex geodesic, $\tilde k_{\mathcal R_{II}}(a,0) = \tilde k_{\mathcal R_{II}} (f(\sigma), 0)= \rho(\sigma, 0).$ Thus $\sigma = \lambda_1$. 

Therefore $f_{11}(\lambda_1)=\lambda_1$ and $f_{11}(0)=0$ so $f_{11}(\lambda)=\lambda$, $\lambda\in \DD$ by the classical Schwarz lemma.

Since $f(\lambda)\in \overline{\mathcal R_{II}}$ for any $\lambda\in \bar\DD$ we find that $|f_{11}|^2 + |f_{12}|^2\leq 1$ on $\bar \DD$. Clearly $|f_{11}^*|=1$ on $\TT$, so $f_{12}\equiv 0$. In the same way we infer that $f_{21}\equiv 0$. This finishes the proof.

\end{proof}


\begin{thebibliography}{10}

\bibitem{Abate} \textsc{M.~Abate}, \textit{The complex geodesics of non-hermitian symmetric spaces}, Universiti degli Studi di Bologna, Dipartamento di Matematica, Seminari di geometria, 1991-1993, 1--18; the paper available at \verb'www.dm.unipi.it/~abate/articoli/artric/files/CompGeodHermSymSp.pdf'.

\bibitem{Abo 2007} \textsc{A.~A.~Abouhajar}, \textit{Function theory related to $H^{\infty}$ control}, Ph.D. thesis, Newcastle University, 2007.

\bibitem{AWY} \textsc{A.~A.~Abouhajar, M.~C.~White, N.~J.~Young}, \textit{A Schwarz lemma for a domain related to  mu-synthesis}, Journal of Geometric Analysis, 17(4), 2007, 717-750.

\bibitem{Agl-Lyk-You 2013} \textsc{J.~Agler, Z.~Lykova, N.~J.~Young}, \textit{Extremal holomorphic maps and the symmetrized bidisc}, Proc. London Math. Soc. (3) 106 (2013) 781-818.

\bibitem{Agl-Lyk-You2 2013} \textsc{J.~Agler, Z.~Lykova, N.~J.~Young}, \textit{$3$-extremal holomorphic maps and the symmetrised bidisc}, J. Geom. Analysis, to appear (2013).

\bibitem{Agl-You 2004} \textsc{J.~Agler, N.~J.~Young}, \textit{The hyperbolic geometry of the symmetrized bidisc},  J. Geom. Anal.  14  (2004),  no. 3, 375--403.

\bibitem{Ama-Tho1994} \textsc{E.~Amar, P.~J.~Thomas}, \textit{A notion of extremal discs related to interpolation in the Ball}, Math. Annalen, 300 (1994) 419-433.

\bibitem{Agl} {\sc J.~Agler, N.~J.~Young}, {\it The two-by-two spectral Nevanlinna-Pick problem}, Trans. Amer. Math. Soc. 356 (2004) 573--585.





\bibitem{Cos 2004b} \textsc{C.~Costara}, \textit{The symmetrized bidisc and Lempert's theorem}, Bull. London Math. Soc. 36 (2004), no. 5, 656--662.



\bibitem{Edi1995} \textsc{A.~Edigarian}, \textit{On extremal mappings in complex ellipsoids},
Ann. Polon. Math. 62 (1995), 83--96.

\bibitem{Edi-Kos-Zwo} \textsc{A.~Edigarian, \L. Kosi\'nski, W.~Zwonek}, \textit{The Lempert Theorem and the Tetrablock}, Journal of Geom. Anal., 23 (2013), no. 4, 1818-–1831.



\bibitem{Har} {\sc L. A. Harris}, {\it Bounded symmetric homogeneous domains in infinite dimensional spaces}, Proceedings on Infinite Dimensional Holomorphy, 13–-40. Lecture Notes in Math., Vol. 364, Springer, Berlin, 1974. 

\bibitem{Hua} \textsc{L.K.~Hua}, \textit{Harmonic Analysis of Functions of Several Complex Variables in the Classical Domains}, AMS. (1963), Providence.



\bibitem{Jar-Pfl 1993} \textsc{M.~Jarnicki, P. Pflug}, \textit{Invariant Distances and Metrics in Complex Analysis}, De Gruyter Expositions in Mathematics 9, 1993.

\bibitem{Jar-Pfl 2000} \textsc{M.~Jarnicki, P. Pflug}, \textit{Extension of Holomorphic Functions}, de Gruyter Expositions in Mathematics 34, Walter de Gruyter, 2000.

\bibitem{Kos 2009} \textsc{\L.~Kosi\'nski}, \textit{Geometry of quasi-circular domains and applications to tetrablock}, Proc. Amer. Math. Soc. 139 (2011), 559--569.


\bibitem{Lem 1981} \textsc{L.~Lempert}, \textit{La m\'etrique de Kobayashi et la repr\'esentation des domaines sur la boule}, Bull. Soc. Math. France 109 (1981), no. 4, 427--474.

\bibitem{Lem 1982} \textsc{L.~Lempert}, \textit{Holomorphic retracts and intrinsic metrics in convex domains}, Anal. Math. 8 (1982), no. 4, 257--261.






\bibitem{Pick} {\sc G.~Pick}, {\it \"Uber die Beschr\"ankungen analytischer Funktionen, welche durch vorgegebene
Funcktionswerte bewirkt werden}, Math. Ann. 77 (1916) 7–23.



\bibitem{You 2007} \textsc{N.~J.~Young}, \textit{The automorphism group of the tetrablock}, Journal of the London Mathematical Society  77(3) (2008), 757-770.

\bibitem{Two} {\sc P.~Tworzewski}, {\it Intersections of analytic sets with linear subspaces}, Annali della Scuola Normale Superiore di Pisa - Classe di Scienze, Sér. 4, 17 no. 2 (1990), p. 227-271.



\bibitem{Zwo 2013} \textsc{W.~Zwonek}, \textit{Geometric properties of the tetrablock}, Arch. Math. 100(2) (2013), 159--165.

\end{thebibliography}
\end{document}